\documentclass{article}
\usepackage{graphicx, amssymb, enumerate, pstricks, pst-node}
\usepackage{pst-all}
\usepackage{url}
\usepackage{lineno}
\usepackage{latexsym}
\usepackage{amsmath,amsthm}
\usepackage{epsfig}
\usepackage{epic}
\usepackage{amssymb}
\usepackage{enumitem}
\usepackage{calc}
\usepackage{authblk}
\newcommand{\ba}{\backslash}
\newcommand{\reals}{\mathbb{R}}
\DeclareMathOperator{\cl}{cl}
\DeclareMathOperator{\tw}{tw}
\DeclareMathOperator{\nw}{nw}
\DeclareMathOperator{\rd}{rd}

\newtheorem{theorem}{Theorem}[section]
\newtheorem{corollary}[theorem]{Corollary}

\newtheorem{lemma}[theorem]{Lemma}

\theoremstyle{definition}

\newtheorem*{example}{Example}

\title{On zeros of the characteristic polynomial of matroids of bounded tree-width}

\author[3]{Carolyn Chun\thanks{chun@usna.edu}}
\author[1]{Rhiannon Hall\thanks{rhiannon.hall@brunel.ac.uk}}
\author[2]{Criel Merino\thanks{merino@matem.unam.mx. Investigaci\'on realizada gracias al Programa UNAM-DGAPA-PAPIIT IN102315.}}
\author[4]{Steven Noble\thanks{s.noble@bbk.ac.uk}}

\affil[1]{Department of Mathematical Sciences, Brunel University London, Uxbridge,  UB8 3PH, United Kingdom}
\affil[2]{Instituto de Matem\'aticas\\ Universidad
    Nacional Aut\'onoma de M\'exico \\ M\'exico City, M\'exico}
\affil[3]{Mathematics Department, United States Naval Academy, Annapolis, MD, United States of America}
\affil[4]{Department of Economics, Mathematics and Statistics, Birkbeck, University of London, London WC1E 7HX, United Kingdom}

\begin{document}

\maketitle

{\small
\copyright 2016
This manuscript version is made available under the CC-BY-NC-ND 4.0 license \texttt{http://creativecommons.org/licenses/by-nc-nd/4.0/}

\texttt{http://dx.doi.org/10.1016/j.ejc.2016.08.011}}

\begin{abstract}
We develop some basic tools to work with representable matroids of bounded tree-width and use them to prove that, for any prime power $q$ and constant $k$, the characteristic polynomial of any loopless, $GF(q)$-representable matroid with tree-width $k$ has no real zero greater than $q^{k-1}$.
\end{abstract}


\section{Introduction}
\label{sec:intro}

For a graph $G$, the chromatic polynomial $\chi_G(\lambda)$ is an invariant which counts the number of proper colourings of $G$ when evaluated at a non-negative integer $\lambda$. However, the chromatic polynomial has an additional interpretation as the zero-temperature antiferromagnetic Potts model of statistical mechanics. 
This has motivated research into the zeros of the chromatic polynomial by theoretical physicists as well as
mathematicians.
Traditionally, the focus from a graph theory perspective has been the positive integer roots, which correspond to the graph not being properly colourable with $\lambda$ colours.
A growing body of work has begun to emerge in recent years more concerned with the behaviour of real or complex roots of the chromatic polynomial.
Sokal~\cite{Sok04} proved that the set of roots of chromatic polynomials is dense in the complex plane.
In contrast, many other results show that certain regions are free from zeros.
For planar graphs, the Birkhoff--Lewis theorem states that the interval $[5,\infty)$  is free from zeros.
For more results along these lines, we direct the reader to the work of
  Borgs~\cite{Bor06},  Jackson~\cite{Jac93},  Sokal \cite{Sok01b}, Thomassen~\cite{Tho97} and  Woodall~\cite{Woo97}.
Perhaps one of the most compelling open questions concerning real zeros is to determine tight bounds on the largest real zero of the chromatic polynomial.
One such bound is given in~\cite{Sok01b} and depends on the maximum vertex degree.
For recent surveys see~\cite{Royle} and~\cite{Dong}.

In matroids, the corresponding invariant is the characteristic polynomial.
The \emph{characteristic polynomial} of a loopless matroid $M$, with ground set $E$ and rank function $r$, is defined by
 \begin{equation*}
     \chi_{M}(\lambda) = \sum_{F\in L} \mu_{M}(\emptyset,F)\lambda^{r(E)-r(F)},
 \end{equation*}
where $L$ denotes the lattice of flats of $M$ and $\mu_{M}$ the M\"obius function of $L$.
When $M$ has a loop, $\chi_{M}(\lambda)$ is defined to be zero.
Observe that for a loopless matroid $M$, $\chi_{M}(\lambda)$ is monic of degree $r(E)$ and that $M$ and its simplification have the same characteristic polynomial.

The projective geometry of rank $r$ over $GF(q)$ is denoted by $PG(r-1,q)$, and $U_{r,n}$, where $n\geq r$, denotes the uniform matroid with rank $r$ containing $n$ elements.
In the uniform matroid, every set of $r$ or fewer elements is independent.
The characteristic polynomials of $PG(r-1,q)$ and $U_{r,n}$ play important roles in this paper, and these are easily computed.
For a prime power $q$, the projective geometry $PG(r-1,q)$ has lattice of flats  isomorphic to the lattice of subspaces of the $r$-dimensional vector space over $GF(q)$.  Hence it has characteristic polynomial
\begin{equation}
\label{eq:PGchar}
\chi_{PG(r-1,q)}(\lambda)=(\lambda-1)(\lambda-q)(\lambda-q^2)\cdots(\lambda-q^{r-1}).
\end{equation}
The largest root of the characteristic polynomial for a projective geometry is therefore $q^{r-1}$.
The characteristic polynomial of the uniform matroid, $U_{r,n}$, is
\[
\chi_{U_{r,n}}(\lambda)=\sum_{k=0}^{r-1} (-1)^{k}\binom{n}{k} (\lambda^{r-k}-1).
\]

For more background on matroid theory, we suggest that the reader consults~\cite{Oxley}.  For the theory of the M\"obius function and the characteristic polynomial, we recommend~\cite{BrOx,Zaslavsky}.

Perhaps the most compelling open question concerning real zeros in this context is deciding whether
there is an upper bound for the real roots of the characteristic polynomial of any matroid belonging to a specified minor-closed class.
Welsh conjectured that no cographic matroid has a characteristic polynomial with a root in $(4,\infty)$.
This was recently disproved by Haggard et al. in~\cite{HPR10}, and, in~\cite{JS}, Jacobsen and  Salas showed that there are cographic matroids whose characteristic polynomials have roots exceeding five.
Consequently, determining whether an upper bound exists for the roots of the characteristic polynomials of cographic matroids remains open.
In~\cite{Royle}, Royle conjectured that for any minor-closed class of $GF(q)$-representable matroids, not including all graphs, there is a bound on the largest real root of the characteristic polynomial.
Given the situation with cographic matroids, this is clearly a difficult conjecture to resolve in the affirmative. In contrast, the situation with graphic matroids has been resolved.
Thomassen~\cite{Tho97} noted that by combining a result that he and Woodall~\cite{Woo97} had obtained independently with a result of Mader~\cite{Mad67}, one obtains the following.

\begin{theorem}
Let $\mathcal F$ be a proper minor-closed family of graphs. Then there exists $c\in \reals$ such that the chromatic polynomial of any loopless graph $G$ in $\mathcal F$ has no root larger than $c$.
\end{theorem}

For certain minor-closed families of graphs, one can find the best possible constant $c$. One such example is the class of graphs with bounded tree-width, a concept originally introduced by Robertson and Seymour~\cite{RS2}.
A \emph{tree-decomposition} of a graph $G$ comprises a tree $T$ and a collection $\{X_t\}_{t\in V(T)}$ of subsets of $V(G)$ satisfying the following properties.
\begin{enumerate}
\item For every edge $uv$ of $G$, there is a vertex $t$ of $T$ such that $\{u,v\} \subseteq X_t$.
\item If $p$ and $r$ are distinct vertices in $T$, the vertex $v$ is in $X_p \cap X_r$ and $q$ lies on the path from $p$ to $r$ in $T$, then $v\in X_q$.
\end{enumerate}
The width of a tree-decomposition is $\max_{t \in V(T)} |X_t|-1$ and the tree-width of a graph is the minimum width of all of its tree-decompositions. As its name suggests, graph tree-width measures how closely a graph resembles a tree.
Matroid tree-width, which we will define later, measures how closely a matroid resembles a tree.
If a graph can be obtained by gluing small graphs together in a tree-like structure, then it has small tree-width.
Likewise, if a matroid can be obtained by gluing small matroids together along a tree-like pattern, then it has small matroid tree-width.
Thomassen~\cite{Tho97} proved the following.
\begin{theorem}
\label{graphy}
For positive integer $k$, let $G$ be a graph with tree-width at most $k$.
Then the chromatic polynomial, $\chi _G(\lambda)$, is identically zero or else $\chi _G(\lambda)>0$ for all $\lambda >k$.
\end{theorem}
Thomassen's proof proceeded essentially as follows, using induction on the number of vertices of $G$.
Let $G$ have tree-width $k$.
Take a tree-decomposition of width $k$, with notation as above. Choose $s$ and $t$ to be neighbouring vertices in $T$. Then $X_s\cap X_t$ is a vertex-cut of $G$. One may add edges to $G$ with both end-vertices in $X_s\cap X_t$ until $X_s\cap X_t$ forms a clique without altering the tree-width. Call this new graph $G'$. The chromatic polynomial of $G$ may be written in terms of the chromatic polynomial of graphs with fewer vertices than $G$ having tree-width at most $k$ and the chromatic polynomial of $G'$ in such a way that one may apply induction provided the result can be established for $G'$. But since $G'$ has a clique whose vertices comprise a vertex-cut, the chromatic polynomial of $G'$ may also be expressed in terms of the chromatic polynomials of graphs with fewer vertices and having tree-width at most $k$.

In this paper, we generalize Thomassen's  useful technique to matroids.
The $GF(q)$-representable matroid analogue of a clique is a projective geometry over $GF(q)$.
A given simple graph $G$ sits inside a clique on $V(G)$ in the same way that a simple $GF(q)$-representable matroid $M$ with rank $r$ sits inside $PG(r-1,q)$.
In the above technique, edges are added to an ``area" of $G$ to form a clique restriction, so that the altered graph has a clique vertex-cut.
This can be viewed as adding edges from the clique on $V(G)$ to the graph $G$ to obtain a clique, across which our graph may be broken.
In this paper, we show how to add elements from $PG(r-1,q)$ to a certain ``area" of $M$ in order to get a $GF(q)$-representable matroid with a certain projective geometry restriction, across which our matroid may be broken.
The map that we use to break apart a matroid is a \emph{tree-decomposition}, which was established by Hlin\u{e}n\'{y} and Whittle in~\cite{HW}.
They developed a matroid analogue of graph tree-width, which we define formally in Section~\ref{tree}.

In order to generalize Thomassen's technique, we first develop some tools for $GF(q)$-representable matroids of bounded \emph{matroid tree-width}.
We then apply these tools to extend his argument to matroid tree-width, which we shall refer to simply as \emph{tree-width} when the context is clear.
In this way, we demonstrate the utility of these tools and simultaneously make progress towards Royle's conjecture.

An alternative way to prove Theorem~\ref{graphy} is to combine the observation that every graph with tree-width at most $k$ has a vertex of degree at most $k$ with Lemma~\ref{oxco} below, established by Oxley for matroids and rediscovered for the special case of graphs by Thomassen~\cite{Tho97} and Woodall~\cite{Woo97}. We show that this proof technique may also be extended to { representable matroids.
In fact, this technique extends to a slightly more general class of matroids, namely matroids that exclude long line minors, which are considered in Theorem~\ref{nolines}.}

It was shown in~\cite{HW} that the tree-width of a matroid is at least equal to the tree-width of each of its minors, thus the class of matroids with tree-width at most $k$ is closed under taking minors.
The following result for such a minor-closed class is the main result of this paper.

\begin{theorem}
\label{maincourse}
For prime power $q$ and positive integer $k$, let $M$ be a $GF(q)$-representable matroid with tree-width at most $k$.
Then $\chi _M(\lambda)$ is identically zero or else $\chi _M(\lambda)>0$ for all $\lambda > q^{k-1}$.
\end{theorem}

{ In the case that $r(M)\leq k$, Theorem~\ref{maincourse} follows easily from known results, in particular Equation~\eqref{eq:PGchar}}. However, this case is not especially interesting, because the rank of a matroid is always bounded below by its tree-width. Our result gives a new bound for representable matroids with high rank and low tree-width.

 The requirement of representability is essential to the result.
For instance, the characteristic polynomial of the $n$-point line, $U_{2,n}$, has a root at $n-1$.
As $U_{2,n}$ has tree-width at most two, the $n$-point lines and their minors form a minor-closed class of matroids with bounded tree-width that do not have an upper bound for the roots of their characteristic polynomials.
Furthermore, the projective geometry $PG(k-1,q)$ has tree-width $k$ and its characteristic polynomial has a root at $q^{k-1}$, hence the bound given is the best possible.
Lemma~\ref{twprop} contains the basic results on tree-width necessary to justify these observations.

{ Given that the line $U_{2,n}$ is the simplest counter-example that we know of, it is natural to consider whether $GF(q)$-representability is necessary, or if excluding a long line minor from a matroid with bounded tree-width is sufficient to bound the roots of the characteristic polynomial, as suggested by Geelen and Nelson~\cite{gnpc}.
We show that this condition is indeed sufficient in the following theorem.
Note that, if $q$ is a prime power, then $U_{2,2+q}$ is an excluded minor for matroids representable over $GF(q)$.
Thus the following theorem applies to a more general class of matroids than Theorem~\ref{maincourse} applies to, although the bound is different.

\begin{theorem}
\label{nolines}
For an integer $q$ at least two, let $M$ be a matroid with tree-width at most $k$ and no minor isomorphic to $U_{2,2+q}$.
Then $\chi _M(\lambda)$ is identically zero or else $\chi _M(\lambda)>0$ for all $\lambda > \frac{q ^k-1}{q -1}$.
\end{theorem}}

{
Combining Theorem~\ref{nolines} with the observation that the characteristic polynomial of $U_{2,n}$ has a root at $n-1$ yields the following dichotomy.

\begin{corollary}\label{dichotomy}
Let $\mathcal M$ be a minor-closed class of matroids having tree-width at most $k$. Then either $\mathcal M$ contains all simple matroids of rank two, or there exists $\lambda_{\mathcal M}$ such that
for any loopless matroid $M$ in $\mathcal M$, $\chi _M(\lambda)>0$ for all $\lambda > \lambda_{\mathcal M}$.
\end{corollary}}

\section{The characteristic polynomial}
\label{sec:prelim}

The characteristic polynomial satisfies many identities similar to those satisfied by the chromatic polynomial.
The following is one such identity, which is particularly important for us.

 \begin{theorem}
\label{condel}
If $e$ is an element of a matroid $M$ that is neither a loop nor a coloop, then the characteristic polynomial of $M$ satisfies
\[
\chi_{M}(\lambda)=\chi_{M\setminus e}(\lambda)-\chi_{M/ e}(\lambda).
\]
\end{theorem}

From Theorem~\ref{condel}, it is easy to see that a loopless matroid and its simplification have the same characteristic polynomial.
The second identity which we will need is a special case of a result of Brylawski~\cite{Brylawski}.
We first define the generalized parallel connection of two matroids $M_1$ and $M_2$ with ground sets $E_1$ and $E_2$, respectively, according to~\cite[page~441]{Oxley}.

Let $T= E_1\cap E_2$ and suppose that $M_1|T = M_2|T$.
Furthermore, suppose that $\cl_{M_1}(T)$ is a modular flat of $M_1$ and that each element of $\cl_{M_1} (T) \setminus T$ is either a loop or parallel to an element of $T$.
Let $N$ denote the common restriction $M_1 |T = M_2 |T$.
Then the \emph{generalized parallel connection across $N$} is the matroid $P_N (M_1 , M_2)$ whose flats are precisely the subsets $F$ of $E_1\cup E_2$ such that $F \cap E_1$ is a flat of $M_1$ and $F \cap E_2$ is a flat of $M_2$.

Suppose a graph $G$ has vertex set $V$ and edge set $E$, where $G=(V,E)=(V_1\cup V_2,E_1\cup E_2)$, such that $G_1=(V_1,E_1)$ and $G_2=(V_2,E_2)$ are themselves graphs.
It is a well-known result that, if the graph $(V_1\cap V_2,E_1\cap E_2)$ is isomorphic to $K_k$, the complete graph on $k$ vertices, then the chromatic polynomial $P_G(\lambda)$ is equal to $\frac{P_{G_1}(\lambda) P_{G_2}(\lambda)}{P_{K_k}(\lambda)}$.
We now state Brylawski's result which generalizes this result to matroids.

 \begin{theorem}[Brylawski (1975)]
\label{bry}
 Let $M$ be a generalized parallel connection of the matroids $M_1$ and $M_2$ across the modular flat $N$. Then
\[
\chi_{M}(\lambda)=\frac{\chi_{M_1}(\lambda)\chi_{M_2}(\lambda)}{\chi_{N}(\lambda)}.
\]
 \end{theorem}


\section{Tree-decompositions}
\label{tree}

This section is devoted to defining matroid tree-width and developing some techniques for considering matroids of bounded tree-width.

A \emph{tree-decomposition} of a matroid $M$ is a pair $(T,\tau)$, where $T$ is a tree and $\tau :E(M)\rightarrow V(T)$ is an arbitrary mapping.
For convenience, let $V(T)=\{v_1,v_2,\dots ,v_\ell\}$ and let $E_i=\tau ^{-1}(v_i)$ for all $i$ in $\{1,2,\dots ,\ell\}$.
We say that $E_i$ is the \emph{bag corresponding to $v_i$}.
Let $c_i$ be the number of components in $T- v_i$ and let $T_{i,1},T_{i,2},\dots ,T_{i,c_i}$ denote the components in $T- v_i$.
For $j\in\{1,2,\dots ,c_i\}$, let $B_{i,j}$ be the subset of $E(M)$ given by $\{e|\tau (e)\in V(T_{i,j})\}$.
The vertex $v_i$ is said to \emph{display} the subsets $B_{i,1},B_{i,2},\dots ,B_{i,c_i}$ of $E(M)-E_i$.
Note that these subsets are pairwise disjoint.
We say that the \emph{rank defect} of $B_{i,j}$, denoted $\rd(B_{i,j})$, is equal to $r(M)-r(E(M) - B_{i,j})$.
Note that this number is the same as the size of the smallest set $I\subseteq B_{i,j}$ such that all of the elements in $B_{i,j}-I$ are in the closure of $E(M)-B_{i,j}$ in the matroid $M/I$.
Clearly $I$ is an independent set in $M$.
The rank defect is therefore a measure of the amount of rank contributed to $M$ solely by the set $B_{i,j}$.
The \emph{node width of a vertex $v_i$}, written $\nw(v_i)$, is equal to $r(M)-\sum\limits _{j=1} ^{c_i} \rd(B_{i,j})$. Note that in the degenerate case where $|V(T)|=1$, the node width of the single vertex of $T$ is equal to $r(M)$.
The \emph{width of $(T,\tau )$} is the maximum node width of all vertices in $V(T)$.
The \emph{matroid tree-width of $M$}, written $\tw(M)$, is equal to the minimum width of all tree-decompositions of $M$.
We let $v(M)$ be the number of vertices in the smallest tree over all of the tree-decompositions with width equal to the tree-width of $M$.
If $(T,\tau)$ is a tree-decomposition of $M$ with width equal to $\tw (M)$ and if $|V(T)|=v(M)$, then we say that $(T,\tau)$ is a \emph{good tree-decomposition} of $M$.

\begin{figure}[htb]
\center
\includegraphics[scale=1]{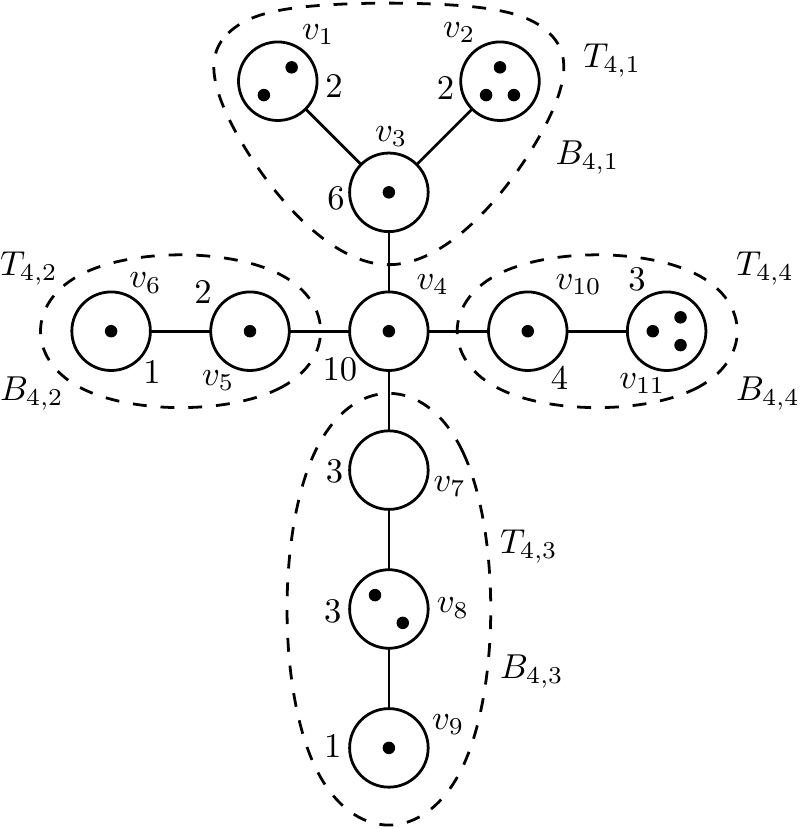}
\caption{A sample tree-decomposition of the uniform matroid $U_{11,16}$.  Each circle is labeled by the vertex of the tree that it represents.  The dots inside each circle represent the matroid elements that are in the bag corresponding to that vertex.  Each circle is also labeled with the node width of its vertex.}
\label{tentacles}
\end{figure}

\begin{example} We give a sample tree-decomposition of $U_{11,16}$ (see Figure~\ref{tentacles}).
Due to the symmetry of the matroid elements, it is not necessary to label the elements of the matroid.
We have illustrated the assignment of elements into bags by placing dots within circles.
Each dot represents an element in $U_{11,16}$ and each circle represents a vertex of the tree in the tree-decomposition.
The vertices of the tree are labeled unambiguously by their names and their node widths.
Each dashed region indicates a subtree of the tree $T$ and these subtrees comprise the connected components of the tree $T\backslash v_4$.
For example, the subtree $T_{4,3}$ consists of the vertex set $\{v_7,v_8,v_9\}$ and edge set $\{v_7v_8,v_8v_9\}$.
As a consequence of each dashed region indicating a connected component of $T\backslash v_4$, the matroid elements within the dashed regions are those of the subsets $B_{4,1},B_{4,2},B_{4,3}$ and $B_{4,4}$ of $E(U_{11,16})- E_4$ displayed by the vertex $v_4$, where $E_4$ is the single-element bag associated with $v_4$.
To compute the node width for $v_4$, note that $\rd (B_{4,1})=1$ and $\rd (B_{4,2})=\rd (B_{4,3})=\rd (B_{4,4})=0$.
Hence $\nw (v_4)=r(U_{11,16})-1=10$.
Note that this is not an optimal tree-decomposition of $U_{11,16}$.
For example, a tree-decomposition whose tree is a path where each bag contains exactly one matroid element has width six.
\end{example}

In addition to that used previously in this section, we employ an alternate use of the term ``display'' as follows.
Let $e=uw$ be an edge of $T$, let $T_u$ and $T_w$ be the two components of $T\ba e$ containing $u$ and $w$ respectively, and let $U$ and $W$ be the sets of matroid elements $U=\left\{ x | \tau(x)\in V(T_u) \right\}$ and $W=\left\{ x | \tau(x)\in V(T_w) \right\}$. We say that the edge $e$ \textit{displays} the sets $U$ and $W$.

We now prove a lemma that will lend some structure to good tree-decompositions, which we establish in the following corollary.

\begin{lemma} \label{red}
Let $(T,\tau)$ be a tree-decomposition of a matroid $M$. Suppose that $T$ has an edge $e=uw$ that displays the sets $U,W\subseteq E(M)$, where $U\subseteq \cl (W)$. Then there exists another tree-decomposition $(T',\tau')$ of $M$ having width at most the width of $(T,\tau)$, such that $|V(T')|<|V(T)|$.
\end{lemma}
\begin{proof}
Consider $T$. Let $T_1$, $T_2$, \dots, $T_\ell$ be the connected components of $T\ba w$, where $u\in T_1$. Note that $U=\tau^{-1}(V(T_1))$. Let $T'$ be $T\ba T_1$.  We define $\tau'$ such that $\tau'(x)=\tau(x)$ if $x\notin U$ and $\tau'(x)=w$ if $x\in U$. Clearly, $|V(T')|<|V(T)|$.
Take $s\in V(T')$. If $s\ne w$ then $s$ displays the same subsets of $E(M)$ in $(T',\tau')$ and $(T,\tau)$, so $\nw_{(T',\tau')}(s)=\nw_{(T,\tau)}(s)$.
%

We conclude this proof by showing that $\nw_{(T',\tau')}(w)=\nw_{(T,\tau)}(w)$.
In the original tree-decomposition, $(T,\tau )$, $w$ displays the subsets $B_1$, $B_2$, \dots, $B_\ell$, where $B_i=\tau^{-1}(V(T_i))$. Note that $B_1=U$.
Whereas in $(T',\tau')$, $w$ displays the subsets $B_2$, $B_3$, \dots, $B_\ell$. Since $B_1=U\subseteq \cl(E(M)-U)$, we have $\rd(B_1)=0$. It follows that $\nw_{(T',\tau')}(w)=\nw_{(T,\tau)}(w)$, as required.
%
\end{proof}

The next result follows immediately from Lemma~\ref{red}.

\begin{corollary}
\label{hall}
Let $M$ be a matroid with tree-width $k$.
If $(T,\tau)$ is a good tree-decomposition of $M$, then, for every pair of subsets $U$ and $W$ of $E(M)$ displayed by an edge of $T$, neither $r(U)$ nor $r(W)$ is equal to $r(M)$.
\end{corollary}

For a good tree-decomposition of a matroid, the preceding result implies that every leaf in the tree corresponds to a set of elements in the matroid that, informally speaking, has some substance.
That is, the set is not in the closure of the rest of the elements in the matroid.
In the corollary following the next lemma, we bound the rank of such a set of elements.

The following lemma is a collection of fundamental results for tree-width and tree decompositions. The first was proved in~\cite{HW}.
\begin{lemma}
\label{twprop}
Let $M$ be a matroid.
Then
\begin{enumerate}[label=(\roman*)]
\item $\tw (M)\geq \tw (N)$ if $N$ is a minor of $M$;
\item $\tw (M)\leq r(M)$, where equality holds if $M$ is a projective geometry; and
\item in any tree-decomposition with tree $T$, the rank of a bag is at most the node width of the corresponding vertex, with equality holding at leaves of $T$.
\end{enumerate}
\end{lemma}

\begin{proof}
Firstly, (i) was proved in~\cite{HW}. For (ii), consider any tree decomposition $(T,\tau)$ of $M$. By definition of node width, no vertex of $T$ can have node width larger than $r(M)$, thus $\tw(M)\leq r(M)$. In the case where $M$ is a projective geometry, to demonstrate that $\tw (M)=r(M)$, it is sufficient to show that a good tree decomposition for $M$ has just one vertex. To that end, suppose that $T$ has an edge that displays $U\subseteq E(M)$ and $W\subseteq E(M)$. Then $\{U, W\}$ partitions $E(M)$. However, for every bipartition of the elements of a projective geometry into sets $U$ and $W$, either $r(U)=r(M)$ or $r(W)=r(M)$, and it follows from Corollary~\ref{hall} that a good tree decomposition for $M$ has just one vertex.

To prove (iii), using the notation set up in our definition of tree width, first note that since $B_{i,1}$,..., $B_{i,c_i}$ are a collection of pairwise disjoint subsets of $E(M)$, by submodularity of the rank function,
\begin{align*}
    r(E(M)-(B_{i,1}\cup B_{i,2})) &=r((E(M)-B_{i,1})\cap(E(M)-B_{i,2})) \\
        & \leq r(E(M)-B_{i,1}) + r(E(M)-B_{i,2})-r(M).
\end{align*}
By repeatedly applying submodularity, we see that $$r(E(M)-(B_{i,1}\cup \dots \cup B_{i,c_i})) \leq \sum\limits _{j=1} ^{c_i} r(E(M)-B_{i,j}) - (c_i-1)r(M).$$ Comparing the rank of the bag of matroid elements $E_{v_i}$ associated with vertex $v_i$ to the node width of $v_i$, we have
\begin{align*}
    r(E_{v_i}) & = r(E(M)-(B_{i,1}\cup \dots \cup B_{i,c_i})) \\
        & \leq \sum\limits _{j=1} ^{c_i} r(E(M)-B_{i,j}) - (c_i-1)r(M) \\
        & = r(M)- \sum\limits _{j=1} ^{c_i} (r(M)- r(E(M)-B_{i,j})) \\
        & = \nw(v_i),
\end{align*}
as required. In the case where $v_i$ is a leaf of $T$, equality holds since $\nw(v_i)=r(M) - \rd(E(M)-E_{v_i}) = r(M)- (r(M)-r(E_{v_i})) = r(E_{v_i})$.
\end{proof}

The next result follows immediately from Lemma~\ref{twprop}.

\begin{corollary}
\label{octopus}
Take $M$ with width-$k$ tree-decomposition $(T,\tau)$.
If $v$ is a vertex of $T$, then the set $\tau ^{-1}(v)$ has rank at most $k$.
\end{corollary}

In Corollaries~\ref{hall} and~\ref{octopus}, we showed that a leaf in the tree of a good tree-decomposition corresponds to a set of elements that has some substance, but not too much substance.
We now find a small cocircuit in the matroid, when it is representable over a finite field.

\begin{lemma}
\label{sharktank}
Let $M$ be a simple $GF(q)$-representable matroid for some prime power $q$ and let $M$ have tree-width $k$ for some positive integer $k$.
Then $M$ has a cocircuit with at most $q^{k-1}$ elements.
\end{lemma}
\begin{proof}
Let $(T,\tau)$ be a good tree-decomposition of $M$. In the case where $v(M)\geq 2$,
$T$ contains a leaf $w$.
Let $E_w=\tau ^{-1}(w)$.
By Lemma~\ref{red}, $E_w$ is not contained in the flat $\cl _M(E(M)-E_w)$.
Hence this flat is contained in a hyperplane of $M$, whose complement is contained in $E_w$.
Evidently there is a cocircuit $C^*$ contained in $E_w$.
Corollary~\ref{octopus} implies that $E_w$ has rank at most $k$.
As $M$ is $GF(q)$-representable and simple, we know that $E_w$ is a restriction of $PG(k-1,q)$.
The largest cocircuit in $PG(k-1,q)$ is obtained by deleting a hyperplane, which leaves $q^{k-1}$ elements.
Hence $|C^*|\leq q^{k-1}$.

In the case where $v(M)=1$, we have $r(M)=k\geq 1$, thus $M$ is a restriction of $PG(k-1,q)$. With rank at least $1$, $M$ contains a cocircuit, and by the same argument as above, $M$ contains a cocircuit $C^*$ with $|C^*|\leq q^{k-1}$.
\end{proof}

During the remainder of this paper, for a simple $GF(q)$-representable matroid $M$, we denote by $\overline{M}_q$ the projective geometry $PG(r(M)-1,q)$ of which $M$ is a spanning restriction.
If $S \subseteq E(\overline{M}_q)-E(M)$, then let $M^S$ denote the restriction of $\overline{M}_q$ to the elements of $E(M)\cup S$.
Take $(T,\tau)$, a tree-decomposition of $M$.
For edge $uw$ in $T$, let $U'$ and $W'$ be the subsets of $E(M)$ displayed by $uw$, where $\tau ^{-1}(u)\subseteq U'$.
Let $U$ be the subset of elements of $\overline{M}_q$ obtained by taking the closure $\cl _{\overline{M}_q}(U')$, and likewise, let $W=\cl _{\overline{M}_q}(W')$.
We say that the \emph{neck of $uw$ with respect to $\overline{M}_q$}, or simply the \emph{neck of $uw$} when the projective geometry is clear, is the set of elements in $U\cap W$.
Note that the neck of each edge is a projective geometry over $GF(q)$. We say that the \emph{external neck of $uw$ with respect to $\overline{M}_q$}, or simply the \emph{external neck of $uw$} is the intersection of the neck of $uw$ with $E(\overline{M}_q)-E(M)$.


\begin{lemma}
\label{starfish1}
Let $(T,\tau)$ be a tree-decomposition of $M$ with width $\tw(M)$ and let $S$ be a subset of the external neck of an edge of $T$. Then $\tw(M)=\tw(M^S)$.
\end{lemma}
\begin{proof}
Let $uw$ be an edge of $T$ whose neck contains $S$.
Now $M^S$ has a tree-decomposition $(T,\tau')$ obtained from $(T,\tau)$ by letting $\tau '(x)=\tau (x)$ when $x\in E(M)$ and by letting $\tau '(x)=u$ when $x\notin E(M)$.
Thus the decomposition is the same except that we add the elements of $S$ to the bag corresponding to $u$. (We could equally well add them to the bag corresponding to $w$.)

We show that, for each edge of $T$, the corresponding subsets of $E(M)$ and $E(M^S)$ displayed by this edge have the same rank defects, and conclude that $M$ and $M^S$ have the same tree-width.
By the definition of rank defect, if the elements of $S$ were added to a set $B$, then $\rd _{M^S}(B\cup S)=r(M^S)-r_{M^S}(E(M)-B)=r(M)-r_{M}(E(M)-B)=\rd_{M}(B)$.
Hence the rank defect of $B$ in $M$ is equal to the rank defect of $B\cup S$ in $M^S$.
If the elements of $S$ were not added to a set $B$ that is displayed by an edge of $T$, then $S$ is a subset of the closure of $E(M)-B$ in $M^S$ by construction. The rank defect again remains unchanged, as $\rd_{M^S}(B)=r(M^S)-r_{M^S}((E(M)-B)\cup S)=r(M)-r_M(E(M)-B)=\rd_M(B)$.
Therefore each vertex in $T$ has the same node width in $(T,\tau)$ and $(T,\tau ')$.
It follows that $\tw(M^S)=\tw(M)$.
\end{proof}

\begin{lemma}
\label{starfish2}
Let $M$ be a simple $GF(q)$-representable matroid with tree-decomposition having tree $T$. Let
$uw$ be an edge of $T$ and suppose that $S=\{s_1,s_2,\dots ,s_n\}$ is the external neck of $uw$.
Then,
\begin{equation}
\chi _{M}(\lambda)=\chi _{M^{S}}(\lambda) + \sum _{i=1}^n \chi _{M^{\{s_1,s_2,\dots ,s_i\}}/s_i}(\lambda).
\end{equation}
\end{lemma}
\begin{proof}
By construction, $s_1$ is neither a loop nor a coloop of $M^{s_1}$.
Furthermore, $s_i$ is neither a loop nor a coloop of $M^{\{s_1,s_2,\dots ,s_i\}}$ for all $i\in\{1,2,\dots ,n\}$.
By Theorem~\ref{condel}, $\chi _M(\lambda) =  \chi _{M^{s_1}/s_1}(\lambda) + \chi _{M^{s_1}}(\lambda)$.
By repeated application of Theorem~\ref{condel},
\begin{align*}
\chi _M(\lambda) &= \chi _{M^{s_1}/s_1}(\lambda) + \chi _{M^{s_1}}(\lambda)\\
&=  \chi _{M^{s_1}/s_1}(\lambda)  +\chi _{M^{\{s_1,s_2\}}/s_2}(\lambda) + \chi _{M^{\{s_1,s_2\}}}(\lambda)\\
&=  \chi _{M^{s_1}/s_1}(\lambda)  +\chi _{M^{\{s_1,s_2\}}/s_2}(\lambda)  +\chi _{M^{\{s_1,s_2,s_3\}}/s_3}(\lambda) + \chi _{M^{\{s_1,s_2,s_3\}}}(\lambda)\\
 &\mathrel{\makebox[\widthof{=}]{\vdots}} \\
&=  \chi _{M^{s_1}/s_1}(\lambda)  +\chi _{M^{\{s_1,s_2\}}/s_2}(\lambda) + \dots  +\chi _{M^S/s_n}(\lambda) + \chi _{M^S}(\lambda).
\end{align*}
Thus, the lemma holds.
\end{proof}

\section{Bounds for zeros of the characteristic polynomial}

In this section we prove the main theorem in two ways, illustrating different techniques each time. The first proof requires us to consider separately the case where $v(M)=1$.
\begin{lemma}
\label{bullseye}
Let $M$ be a loopless, $GF(q)$-representable matroid of tree-width $k$, for some prime power $q$ and some positive integer $k$, with $v(M)=1$.
Suppose that, if $N$ is a loopless, $GF(q)$-representable matroid with tree-width at most $k$ and $r(N)<r(M)$, then $\chi _N(\lambda)>0$ for all $\lambda >q^{k-1}$.
Then $\chi _{M}(\lambda)>0$ for all $\lambda >q^{k-1}$.
\end{lemma}
\begin{proof}We may assume that $M$ is simple.
Let $(T,\tau)$ be a good tree-decomposition of $M$.
The single vertex in $V(T)$ must have node width $k$. By Lemma~\ref{twprop}, we have $k=r(M)$.
Let $S$ be the set of elements in $E(\overline{M}_q)-E(M)$.
As in the proof of Lemma~\ref{starfish2}, by repeated application of Theorem~\ref{condel},
$\chi _M(\lambda)=  \chi _{M^{s_1}/s_1}(\lambda)  +\chi _{M^{\{s_1,s_2\}}/s_2}(\lambda) + \dots  +\chi _{M^S/s_n}(\lambda) + \chi _{M^S}(\lambda)$.
By assumption, each term of this sum is positive for all $\lambda>q^{k-1}$ with the possible exception of $\chi _{M^S}(\lambda)$.
As $M^S$ is a projective geometry with rank $r(M)=k$,
it follows that
$\chi _{M^S}(\lambda)=(\lambda-1)(\lambda-q)(\lambda-q^2)\cdots(\lambda-q^{k-1})$.
Thus $\chi _{M^S}(\lambda)>0$ for all $\lambda >q^{k-1}$.
\end{proof}

The first proof of the main theorem uses basic tools from characteristic polynomials, and exemplifies the tree-decomposition techniques established by Hlin\u{e}n\'{y} and Whittle in~\cite{HW}, and further developed in this paper, to generalize Thomassen's graph technique.

\begin{proof}[Proof of Theorem~\ref{maincourse}]
If $M$ has a loop, then its characteristic polynomial is identically zero, so we may assume that $M$ is loopless.
As $M$ and its simplification have the same characteristic polynomial and the same tree-width, we may assume that $M$ is simple.
We proceed by induction on $r(M)$.
Suppose that $r(M)=1$.
Then $M\cong U_{1,1}$ and
 $\chi _M(\lambda)=\lambda -1$.
Thus $\chi _M(\lambda)>0$ if $\lambda>1$, hence $\chi _M(\lambda)$ is certainly strictly positive for all $\lambda >q^{k-1}$.

We now assume $r(M)>1$.
Take $(T,\tau)$, a good tree-decomposition of $M$.
If $T$ has a single vertex, then by Lemma~\ref{bullseye}, the result follows.
Thus, we may assume that $T$ contains a leaf $w$ with neighbour $u$. Let $S=\{ s_1,s_2,\dots ,s_n\}$ be the elements in the external neck of $uw$.
By Lemma~\ref{starfish1}, $\tw (M)=\tw (M^S)$ and by Lemma~\ref{starfish2}
\begin{equation}\label{fishy}\chi _M(\lambda) =  \chi _{M^{S}}(\lambda)+\sum _{i=1}^n \chi _{M^{\{s_1,s_2,\dots ,s_i\}}/s_i}(\lambda).\end{equation}

Since $M$ and, consequently, $M^S$ are simple, each of the matroids appearing in the sum on the right-hand side of~\eqref{fishy} is loopless and has rank $r(M)-1$.
Lemma~\ref{twprop} implies that tree-width is not increased by contracting elements.
By induction the characteristic polynomial of $M^{\{s_1,s_2,\dots ,s_i\}}/s_i$ is strictly positive for all $\lambda >q^{k-1}$, for all $i\in\{1,2,\dots ,n\}$.

It remains to consider $\chi _{M^S}(\lambda)$.
Let $S'$ be the neck of $uw$, which is contained in $M^S$.
Clearly $M^S|S'\cong PG(r'-1,q)$ for some $r'$.
Let $E_w$ be the bag corresponding to $w$.
Let $M_1=M^S|(E_w\cup S')$ and $M_2=M^S\ba (E_w-S')$.
Then $M_1|S'=M_2|S'$.
By~\cite[Corollary~6.9.6]{Oxley}, $S'$ is a modular flat in $M_1$.
By~\cite[Proposition~11.4.15]{Oxley},
$M^S$ is the generalized parallel connection of $M_1$ and $M_2$ across $M_1|S'$.
Since $M$ has tree-width at most $k$, we know that
\[ r'=r_{M^S}(S') \leq r_{M^S}(E_w\cup S') = r_{M}(E_w) \leq k,\]
with the last part following from Corollary~\ref{octopus}.
Thus, by Theorem~\ref{bry},
\[
\chi _{M^S}(\lambda)=\frac{\chi _{M_1}(\lambda)\chi _{M_2}(\lambda)}{\chi _{PG(r'-1,q)}(\lambda)}.
\]

Using Equation~\eqref{eq:PGchar}, we see that the denominator is strictly positive for all $\lambda>q^{r'-1}$. Hence it is strictly positive for all $\lambda >q^{k-1}$.
By Corollary~\ref{hall}, since $T$ has $v(M)=v(M^S)$ vertices,  both $M_1$ and $M_2$ have rank less than $r(M^S)=r(M)$.
By our inductive hypothesis, both $\chi_{M_1}(\lambda)$ and $\chi_{M_2}(\lambda)$ are strictly positive for all $\lambda >q^{k-1}$. Thus $\chi_{M^S}(\lambda) > 0$ for all $\lambda >q^{k-1}$, as required.

 \end{proof}


It is also possible to generalize the second proof of Theorem~\ref{graphy}, outlined in the introduction, to matroids representable over a finite field by using the following result of Oxley~\cite[Lemma~2.7]{ox77}.

\begin{lemma}
\label{oxco}
Let $C^*=\{x_1,x_2,\dots ,x_m\}$ be a cocircuit of $M$.
Let $X_{i,j}=\{x_1,x_2,\dots ,x_{i-1},x_{i+1},\dots ,x_{j-1}\}$ for all $1\leq i<j\leq m$.
Then
\[
\chi _M(\lambda)=(\lambda -m)\chi _{M\backslash C^*}(\lambda) +\sum ^m _{j=2}\sum ^{j-1} _{i=1} \chi _{M\backslash X_{i,j}/x_i,x_j}(\lambda).
\]
\end{lemma}

{
A minor-closed family of matroids $\mathcal{M}$ has the \emph{bounded cocircuit property} if there is a constant $f(\mathcal{M}) =f$
{
such that any simple matroid $M$ in $\mathcal{M}$ has a cocircuit of size at most $f$. We now apply Lemma~\ref{oxco} to any minor-closed family of matroids with the bounded cocircuit property.
\begin{lemma}\label{le:bdedcoc}
Let $\mathcal M$ be a minor-closed family of matroids having the bounded cocircuit property with constant $f$. Then for any $M$ in $\mathcal M$, either $\chi_M(\lambda)$ is identically zero or $\chi_M(\lambda)>0$ for all $\lambda > f$.
\end{lemma}
\begin{proof}
Let $M$ be a matroid in $\mathcal M$. We may assume that $M$ is simple and
that the result is valid if $r(M)=1$.

We now assume $r(M)>1$ and proceed using induction on $r(M)$.
Because $\mathcal M$ has the bounded cocircuit property, we know that $M$ has a cocircuit $C^*$ with size at most $f$.
Let $C^*=\{x_1,x_2,\dots ,x_{|C^*|}\}$ and let $X_{i,j}=\{x_1,x_2,\dots ,x_{i-1},x_{i+1},\dots ,x_{j-1}\}$ for $1\leq i<j\leq |C^*|$.
By Lemma~\ref{oxco}, $\chi _M(\lambda)$ is equal to the following
\begin{equation} \label{eq:cpexp} (\lambda -|C^*|)\chi _{M\backslash C^*}(\lambda) +\sum ^{|C^*|} _{j=2}\sum ^{j-1} _{i=1} \chi _{M\backslash X_{i,j}/x_i,x_j}(\lambda).
\end{equation}
Now $r(M\backslash C^*)=r(M)-1$ and $r(M\backslash X_{i,j}/x_i,x_j) = r(M)-2$, for all $1\leq i<j\leq |C^*|$.
By induction, each of the characteristic polynomials appearing in~\eqref{eq:cpexp}
is either identically zero or strictly positive for $\lambda > f$. Furthermore $M\backslash C^*$ is loopless and so $\chi _{M\backslash C^*}(\lambda)>0$ for $\lambda > f$.
As $|C^*|\leq f$, we conclude that $(\lambda -|C^*|)$, and hence $\chi _M(\lambda)$, is strictly positive for all $\lambda >f$.\end{proof}
}}

We now give the alternate proof of Theorem~\ref{maincourse}.

\begin{proof}[Second proof of Theorem~\ref{maincourse}]
{
Let $\mathcal M$ be the family of $GF(q)$-representable matroids with tree-width at most $k$.
Lemma~\ref{twprop} implies that $\mathcal M$ is a minor-closed class and
Lemma~\ref{sharktank} implies that $\mathcal M$ has the bounded cocircuit property with constant
$q^{k-1}$. The result now follows from Lemma~\ref{le:bdedcoc}.}
\end{proof}

%
%

{
\section{Generalizing to matroids with the bounded cocircuit property}

The argument in the second proof of Theorem~\ref{maincourse} may be extended to any family of matroids with the bounded cocircuit property. We show the family of matroids with tree-width at most $k$ containing no $U_{2,2+q}$ minor is one such family by using the following theorem of Kung~\cite{kung}.}

{
\begin{theorem}
\label{nolineskung}
Let $q$ be an integer at least two.
If $M$ is a simple matroid with rank $r$ having no $U_{2,2+q}$-minor, then $|E(M)|\leq \frac{q ^r-1}{q -1}$.
\end{theorem}

Let $p_q$ be the largest prime less than or equal to $q$.
When $r$ is sufficiently large, Geelen and Nelson showed that the bound on the number of elements can be obtained by replacing $q$ with $p_q$ in the preceding theorem.
In~\cite{Nel}, Nelson conjectured that this improvement holds as long as $r\geq 4$.
If Nelson's conjecture holds, then the bound given in Theorem~\ref{nolines} can be improved by replacing $q$ with $p_q$ in the case that $r(M)\geq 4$.

We now prove Theorem~\ref{nolines}, in which we replace the representability condition of Theorem~\ref{maincourse} with the condition that $M$ contain no long line minor.
This generalization was suggested by Geelen and Nelson~\cite{gnpc}.

\begin{proof}[Proof of Theorem~\ref{nolines}]
{ We claim that if $M$ is simple, has tree-width at most $k$ and has no $U_{2,2+q}$ minor, then it has a cocircuit of size at most $\frac{q ^k-1}{q -1}$. The result then follows by noting that the class of matroids with tree-width at most $k$ having no $
U_{2,2+q}$ minor is a minor-closed class and applying Lemma~\ref{le:bdedcoc}.

Once again, we may assume that $M$ is simple and that the claim is valid if $r(M)=1$.
We proceed by induction on $r(M)$. Take $(T,\tau)$, a good tree-decomposition of $M$ and vertex $v\in V(T)$ with degree at most one.
Let $E_v$ be the bag corresponding to $v$ and let $r=r(E_v)$.
Lemma~\ref{twprop}(iii) implies that $r\leq k$.
If $T$ consists of a single vertex, then $E(M)=E_v$.
Furthermore $E_v$ contains a cocircuit since $r\geq 1$.
Suppose then that $T$ contains more than one vertex.
Then $v$ is a leaf vertex. Since $(T,\tau)$ is a good tree-decomposition, $E_v$ is not contained in $\cl (E(M)-E_v)$ by Lemma~\ref{red}.
Thus $E_v$ contains a cocircuit of $M$.}

Theorem~\ref{nolineskung} implies that $|E_v|\leq \frac{q ^r-1}{q -1}\leq \frac{q ^k-1}{q -1}$.
Thus $M$ has a cocircuit $C^*$ of size at most $\frac{q ^k-1}{q -1}$.
\end{proof}

{ In Corollary~\ref{dichotomy}, we}
{ completely determine whether there is a bound on the largest real root of the characteristic polynomial of any matroid belonging to a minor-closed family having bounded tree-width.}
It would be interesting to find minor-closed classes of matroids that do not have the bounded cocircuit property and determine bounds on the real characteristic roots.}

{
\section*{Acknowledgements}
We thank the anonymous referees, Jim Geelen and Peter Nelson for several helpful comments.}

\end{document}